\documentclass{article} 

\usepackage{graphicx} 
\usepackage{titlesec}
\usepackage{changepage}

\usepackage{booktabs} 
\usepackage{array} 
\usepackage{paralist} 
\usepackage{verbatim} 
\usepackage{subfig} 
\usepackage{enumerate,enumitem}

\usepackage{mathrsfs} 
\usepackage{amssymb}
\usepackage{amsmath}
\usepackage{amsthm}
\usepackage{color}
\usepackage[all,pdf,cmtip]{xy}

\usepackage[backend=bibtex]{biblatex}
\title{Localized Transfunctions}
\author{Jason Bentley and Piotr Mikusi\'{n}ski\\
Department of Mathematics\\
University of Central Florida\\
Orlando, Florida, USA}

\newcommand{\mc}[1]{\mathcal{#1}}
\newcommand{\mb}[1]{\mathbb{#1}}
\newcommand{\mf}[1]{\mathbf{#1}}
\newcommand{\rsq}{\rightsquigarrow}

\newcommand{\eps}{\varepsilon}
\newcommand{\supp}{\operatorname{supp}}
\newcommand{\nulls}{\operatorname{null}}

\newtheorem{theorem}{Theorem}[section]

\newtheorem{proposition}[theorem]{Proposition}
\newtheorem{corollary}[theorem]{Corollary}

\theoremstyle{definition}
\newtheorem{definition}[theorem]{Definition}
\newtheorem{example}[theorem]{Example}
\newenvironment{remark}[1][Remark]{\begin{trivlist}
\item[\hskip \labelsep {\bfseries #1}]}{\end{trivlist}}

\begin{document}
\maketitle

\section*{Abstract}
A transfunction is a function which maps between sets of finite measures on measurable spaces (see \cite{Mikusinski}). 

Push-forward operators form one important class of examples of transfunctions and are identified with their respective measurable functions. In this regard, transfunctions are a generalization of measurable functions between measurable spaces. Additionally, there are naturally arising transfunctions with nice properties which are not measurable functions.

Transfunctions which are weakly $\sigma$-additive (commutable with addition over countable sequences of orthogonal measures) between second-countable metric spaces are of particular interest and are primarily developed in this paper. We study such transfunctions which are localized: sending source measures carried by small open sets to target measures also carried by small open sets. 


With the right settings and assumptions, we develop some theorems which characterize continuous functions and measurable functions, and show that the behavior of localized transfunctions can be approximated by measurable functions and by continuous functions, but only up to some error. We also characterize transfunctions that correspond to Markov operators.

In our investigation of transfunctions we are motivated by several potential applications, including Monge-Kantorovich transportation problem or population dynamics that will be presented in some detail in this paper. 

\section{Introduction}

Let $(X,\Sigma_X)$ and $(Y, \Sigma_Y)$ be measurable spaces with sets of finite measures $\mc{M}_X$ and $\mc{M}_Y$, respectively. A \textit{transfunction} is any function $\Phi : \mathcal{M}_X \to \mathcal{M}_Y$. 

Let $f: (X,\Sigma_X) \to (Y,\Sigma_Y)$ be a measurable function. Then the push forward operator
$f_\# : \mc{M}_X \to \mc{M}_Y $ defined by $f_\# (\mu) (B) := \mu (f^{-1}(B))$ is a transfunction. In particular, $f_\#$ is the push-forward operator associated with $f$. We will now identify the measurable function $f$ with the corresponding transfunction $f_\#$ as defined above. 

With this identification established, we shall treat transfunctions $\Phi: \mc{M}_X \to \mc{M}_Y$ as generalized functions, and notate them as $\Phi: X \rsq Y$ when the context is clear.

One can easily verify that the transfunction $f_\#$ has several of the properties mentioned in \cite{Mikusinski}. These properties are nice for transfunctions to have, and in certain applications, they might be assumed for the transfunction model, but in general, none of these nice properties are guaranteed to transfunctions, just as functions are not guaranteed to have nice properties such as continuity, linearity, injectivity or surjectivity. 

Here we shall discuss two potential directions in which transfunctions might be applied. The first will be the Monge-Kantorovich transportation problem, and the second will be setting up transfunctions as models for population dynamics in mathematical biology.

Ambrosio summarizes the Monge-Kantorovich problem and its several generalizations and strategies in the first section of \cite{Ambrosio}. Notably, the use of transport plans (measures on product spaces with marginals equal to the prior and posterior distributions) as a generalization of transport mappings provides effective weak solutions to the M-K problem, which may correspond to true solutions under the right conditions. We now state the M-K problem as similarly defined in \cite{Ambrosio}.

Let $(X,\Sigma_X) = (Y,\Sigma_Y) = (\mb{R}^d, \mc{R}^d)$ be the usual measurable space ($d=2$ normally) with Lebesgue measure $\lambda$. Let $\mc{M}_\lambda$ denote the space of measures on $\mc{R}^d$ absolutely continuous with respect to $\lambda$. Let $\rho^X, \rho^Y$ be prior and posterior probability measures absolutely continuous with respect to $\lambda$. Let $c: \mb{R}^d \times \mb{R}^d \to [0,\infty)$ be a cost function.

The goal of the Monge-Kantorovich problem is to consider all measurable maps $\mc{F} := \{T: \mb{R}^d \to \mb{R}^d \text{ measurable } : \rho^Y = T_\# \rho^X\}$ which push $\rho^X$ to $\rho^Y$, and to determine if there exists one map among $\mc{F}$ with minimum cost
\begin{align*}
\inf_{T \in \mc{F}} \left \{ \int_X c(x,T(x)) d\rho^X (x) \right\} .
\end{align*}

Instead of transport maps $\mc{F}$, one may also consider the collection $\mc{P}$ of all transport plans, that is, all measures on the product $\sigma$-algebra of $\Sigma_X,\Sigma_Y$ such that  $\rho^X(A) = \mu(A \times Y)$ for all $ A \in \Sigma_X$ and $\rho^Y(B)=\mu(X \times B)$ for all $B \in \Sigma_Y\}$. In this case, the goal is to find a transport plan with minimum cost
\begin{align*}
\inf_{\mu \in \mc{P}} \left \{ \int_{X\times Y} c\; d\mu \right\} .
\end{align*}

Transport plans can be characterized using Markov operators, described in \cite{Taylor}.
In a later section of the paper, we will characterize transport plans via a class of transfunctions with properties analogous with Markov operator properties (strongly $\sigma$-additive, positive and total-measure-preserving). The properties found for these transfunctions are relatively strong, and we are currently investigating whether they can be relaxed to form a more generalized setup than what Markov operators (hence transport plans) have.

Now we discuss potential models for mathematical biology. 

Given measure $\mu$ and $\mu$-integrable $f$, we denote by $f \cdot \mu$ or by $\int_\square f d\mu$ the measure $A \mapsto \int_A f d\mu$.

Let $X=Y=\mb{R}^d$ with Lebesgue measure $\lambda$. Let $\eps>0$ and let $\kappa$ be a measure with support in  $B(0,\eps)$. Let $f: \mb{R}^d \to \mb{R}^d$ be measurable and let $g: \mb{R}^d \to [0,\infty)$ be $\lambda$-integrable. Then consider the transfunction 
$$
\Phi : \mu \mapsto g \cdot ((f_\# \mu) \ast \kappa)
$$ 
which models how a population $\mu$ will migrate or travel via $f$ to become $f_\# \mu$, diffuse by $\kappa$ (from territorial behavior or dispersal of offspring) to become $(f_\# \mu) \ast \kappa$, reproduce or die at rates depending on the environmental factors $g$ (food, water, shelter, predators, etc) to become $g \cdot ((f_\# \mu) \ast \kappa)$ after some set amount of time.

If one is dealing with an immobile species (such as rooted plants), then $f$ is simply the identity and the model becomes $\mu \mapsto g \cdot (\mu \ast \kappa)$, where $\kappa$ describes the dispersal of seeds.

If one is dealing with microorganisms within a predetermined culture, then $g$ may be nearly a constant which contains the expected growth rate, and $f$ may be minuscule or significant depending on whether culture is stationary or agitated.

In all of these examples, $\kappa$, $f$ and $g$ are fixed. This may cause limitations in describing natural phenomena. However, we could allow $\kappa$, $f$ and $g$ to update at each step and choose to iterate the model described above. Alternatively, in future work, we may also consider a family of transfunctions $\{\Phi_t : \mb{R}^d \to \mb{R}^d : t \ge 0\}$ where $\Phi_0 \mu = \mu$ and $\Phi_s(\Phi_t(\mu)) = \Phi_{s+t}~(\mu)$ to describe a continuous model for natural phenomena. Then $\kappa$, $f$ and $g$ could update continuously.

\section{Prerequisite Definitions and Theorems}

Unless otherwise specified, all instantiated measures shall be finite and positive. Occasionally, we may sum countably many measures together. When this occurs, the sum may be finite or infinite and we will not determine the finiteness of the measure whenever it is inconsequential to the argument at hand.

\subsection*{Carriers}

\begin{definition}\label{def_carrier} If $\mu$ is a positive or vector measure on $(X,\Sigma_X)$ and $A\in \Sigma_X$, then we say that \textbf{$A$ is a carrier of $\mu$} and write $\mu\sqsubset A$ if $|\mu|(A^c) = 0$, where $|\mu|$ denotes the variation measure of $\mu$. If $\mu$ is a positive measure, then $\mu \sqsubset A$ is also equivalent to the simpler condition that $\mu(A^c) = 0$. 

\end{definition}
The following proposition shows some useful, basic properties of measures and their carriers.

\begin{proposition}\label{prop_carrier}
Let $\mu$, $\nu$, $\kappa_1, \kappa_2, \dots,$ and $\sum_{i=1}^\infty \kappa_i$ either all be vector measures or all be positive measures in $(X,\Sigma_X)$, and let $A,B,C_1,C_2,\dots$, and $\cap_{i=1}^\infty C_i$ be measurable sets in $\Sigma_X$. Then
\begin{enumerate}[label=(\roman*)]
\item[\rm{(i)}] $\mu \sqsubset A \text{ and } A \subseteq B \text{ implies that } \mu \sqsubset B$;
\item[\rm{(ii)}] $\forall i \in \mb{N} \text{ , } \mu \sqsubset C_i \text{ implies that } \mu \sqsubset \cap_{i=1}^\infty C_i$;
\item[\rm{(iii)}] $|\mu| \le |\nu| \text{ and } \nu \sqsubset A \text{ implies that } \mu \sqsubset A$;
\item[\rm{(iv)}] $\forall i \in \mb{N} \text{ , } \kappa_i \sqsubset A \text{ implies that }  \textstyle{\sum}_{i=1}^\infty \kappa_i \sqsubset A$.
\end{enumerate}
\end{proposition}

\begin{proof}

To prove (i), notice that $|\mu|(A^c)=0$ and $B^c \subseteq A^c$ imply that $|\mu|(B^c) = 0$. Similarly, to prove (ii), observe that $|\mu|\left( \left( \cap_{i=1}^\infty C_i \right)^c \right) = |\mu|\left(\cup_{i=1}^\infty C_i^c \right) \le \sum_{i=1}^\infty |\mu|(C_i^c) = 0$ follows from $\sigma$-subadditivity of $\mu$. Proving (iii) requires verifying $|\mu|(A^c) \le |\nu|(A^c) = 0$. Finally, to prove (iv), show that $|\sum_{i=1}^\infty \kappa_i| (A^c) \le (\sum_{i=1}^\infty |\kappa_i|) (A^c) = \sum_{i=1}^\infty |\kappa_i| (A^c) = 0.$
\end{proof}

Knowing that every measure on $(X,\Sigma_X)$ is carried by $X$, one can see from (i) and (ii) of Proposition \ref{prop_carrier} that carriers of a fixed measure form a filter on the poset $(\Sigma_X, \subseteq)$; in fact, non-zero measures correspond to proper filters and these filters are stronger in the sense that countable intersections are allowed (normally, only finite intersections are allowed).

Also, parts (iii) and (iv) of Proposition \ref{prop_carrier} indicate that all vector measures (over field $\mb{R}$ or $\mb{C}$) carried by a fixed set form a Banach space under the finite total variation norm, and that all positive measures carried by a fixed set form a convex cone closed under countable addition.

\subsection*{Projections of Measures}
\begin{definition}\label{def_projection}
Let $\mu$ be a measure on measurable space $(X,\Sigma_X)$, and let $A \in \Sigma_X$. Then the \textbf{projection of $\mu$ onto $A$}, notated as $\pi_A \mu$, is the measure defined via $\pi_A \mu : B \mapsto \mu(B \cap A)$ for $B \in \Sigma_X$. If $\mc{M}_X$ is a space of measures on $(X,\Sigma_X)$, then we say that $\mc{M}_X$ is \textbf{closed under projections} if $\mu \in \mc{M}_X$ implies that $\pi_A \mu \in \mc{M}_X$ for all $A \in \Sigma_X$. 
\end{definition}

\begin{corollary}\label{cor_projection}
$\mu \sqsubset A$ if and only if $\pi_A \mu = \mu$.
\end{corollary}

\begin{proof}
$\mu \sqsubset A$ implies that $\mu(B) = \mu(A\cap B) + \mu(A^c \cap B) = \mu(A \cap B) + 0 = \pi_A \mu (B)$ for all $B \in \Sigma_X$, showing the forward direction. The equality $\pi_A \mu = \mu$ implies that $|\mu|(A^c) = |\pi_A \mu|(A^c) = 0$, showing reverse direction.
\end{proof}

The space of all finite measures (positive or vector-valued) on $(X,\Sigma_X)$ is easily verified to be closed under projections. However, there are other non-trivial examples of spaces of measures which are closed under projections. One example is the space of all non-atomic (positive) measures, as will be discussed below.

\subsection*{Atomic and Non-Atomic Measures}

We introduce atoms of measures, purely atomic measures and non-atomic measures, as defined and discussed by Johnson in \cite{Johnson}.

\begin{definition}\label{def_atom}
Let $(X,\Sigma)$ be a measurable space, let $\mu$ be a positive measure on that space, and let $A \in \Sigma$. Then \textbf{A is an atom of $\mu$} if $\mu(A)>0$ and for all $\varnothing \subseteq B \subseteq A$, $\mu(B) \in \{0,\mu(A)\}$. A \textbf{purely atomic} measure $\mu$ is a measure where every non-$\mu$-null set contains an atom. A \textbf{non-atomic} measure $\mu$ is a measure with no atoms.
\end{definition}
As a consequence of the definitions, if $\mu$ is non-atomic, then every set $A \in \Sigma$ with $\mu(A) > 0$ admits a proper, non-empty subset $B$ of $A$ with $0 < \mu(B) < \mu(A)$. 

Johnson showed, among other results, that any measure $\mu$ on a $\sigma$-ring is the sum of a purely atomic measure and a non-atomic measure, and that the decomposition is unique when $\mu$ is $\sigma$-finite.

~

\begin{example}\label{ex_atomic}
The Dirac measure $\delta_0$ on $\mb{R}$ with the $\sigma$-algebra of Borel subsets $\mc{B}(\mb{R})$ where $\delta (A) = 1$ if $0 \in A$ and $\delta (A) = 0$ if $0 \notin A$ has atom $\{0\}$. Notice that $\delta_0$ is also purely atomic. The Lebesgue measure $\lambda$ on $(\mb{R}, \mc{B}(\mb{R}))$ is a non-atomic measure because $\lambda$ is strictly-positive, inner-regular and compact sets have finite Lebesgue measure. The zero measure is vacuously purely atomic \emph{and} non-atomic. In our future setting, this will be the only such measure.
\end{example}

~

\begin{proposition}\label{prop_non-atomic_projection}
Let $\mu \in \mc{M}_X$ be non-atomic and let $A \in \Sigma_X$. Then $\pi_A \mu$ is non-atomic.
\end{proposition}

\begin{proof}
If $\pi_A \mu$ is the zero measure, then $\pi_A \mu$ is non-atomic by the previous example. Suppose instead that $\pi_A \mu$ is non-zero. Then $\pi_A \mu (X) = \mu(A\cap X) = \mu(A) > 0$. Since $\mu$ is non-atomic, there exists some set $B \in \Sigma_X$ such that $B \subset A$ and that $0 < \pi_A \mu (B) = \mu(B) < \mu(A) =\pi_A \mu (X)$. Therefore, $\pi_A \mu$ is non-atomic. With  $\mu$ and $A$ arbitrary, the proposition follows.
\end{proof}

\subsection*{Orthogonal Measures}
\begin{definition}\label{def_orthogonal}
If $\mu$, $\nu$ are measures in $\mc{M}_X$, then they are called \textbf{orthogonal}, written as $\mu \perp \nu$, if there exists $A \in \Sigma_X$ such that $\mu \sqsubset A$ and $\nu \sqsubset A^c$. A countable sequence of measures $\{\mu_n\}_{n=1}^\infty$ is called \textbf{(pairwise) orthogonal} if $\mu_i \perp \mu_j$ for $i \not= j$. 
\end{definition}
~

\begin{proposition}\label{prop_orthogonal_sets}
Let $\{\mu_i\}_{i=1}^\infty$ be an orthogonal sequence of measures in $\mc{M}_X$. Then there exists a pairwise disjoint sequence $\{S_i\}_{i=1}^\infty$ of measurable sets in $\Sigma_X$ such that $\mu_i \sqsubset S_i$ for all $i \in \mb{N}$.
\end{proposition}

\begin{proof}
For $i < j$ in $\mb{N}$, let $T_{i,j}$ denote the measurable set with $\mu_i \sqsubset T_{i,j}$ and $\mu_j \sqsubset T_{i,j}^c$ guaranteed by $\mu_i \perp \mu_j$. We then define $T_{j,i} = T_{i,j}^c$. For technicality, define $T_{i,i} := X$. Then define $S_i := \cap_{j=1}^\infty T_{i,j}$ for all $i \in \mb{N}$.

By a Proposition \ref{prop_carrier} (ii), it follows that $\mu_i \sqsubset S_i$ for all $i \in \mb{N}$. Furthermore, for $i\not=j$ in $\mb{N}$, we have $S_i \cap S_j \subseteq T_{i,j} \cap T_{j,i} = T_{i,j} \cap T_{i,j}^c = \varnothing$. Therefore, $\{S_i\}_{i=1}^\infty$ is a pairwise disjoint sequence.
\end{proof}

\begin{proposition}\label{prop_orthogonal_sum}
Let $\{\mu_i\}_{i=1}^\infty$ be an orthogonal sequence of vector measures or of finite positive measures in $\mc{M}_X$ with pairwise disjoint carrier sequence $\{S_i\}_{i=1}^\infty$ guaranteed by Proposition \ref{prop_orthogonal_sets} such that $\sum_{i=1}^\infty ||\mu_i|| < \infty$. Then $\mu := \sum_{i=1}^\infty \mu_i$ is a vector measure or a finite positive measure such that $\pi_{S_i} \mu = \mu_i$ for all $i \in \mb{N}$.
\end{proposition}

\begin{proof}
The set of all vector-valued measures and the set of all finite signed-valued measures each form a Banach space with respect to the total variation norm, so absolute convergence of $\sum_{i=1}^\infty ||\mu_i||$ implies convergence of $\sum_{i=1}^\infty \mu_i$; therefore, $\mu$ is a vector-valued measure or a finite positive measure. For all $i \in \mb{N}$ and $A \in \Sigma_X$, we have that $\pi_{S_i} \mu (A) = \mu(S_i \cap A) = \sum_{n=1}^\infty \mu_n (S_i \cap A) = \mu_i (S_i \cap A) = \mu_i (A)$, which implies that $\pi_{S_i} \mu = \mu_i$ for all $i \in \mb{N}$.
\end{proof}

\begin{definition}\label{def_sum_of_measures}
If a sequence of measures $(\mu_i)_{i=1}^\infty$ satisfies $\sum_{i=1}^\infty ||\mu_i|| < \infty$, then we call the finite measure $\mu := \sum_{i=1}^\infty \mu_i$ the \textbf{bounded sum of $\mathbf{(\mu_i)_{i=1}^\infty}$}.  A bounded sum $\mu = \sum_{i=1}^\infty \mu_i$ with $\{\mu_i\}_{i=1}^\infty$ being orthogonal as in Proposition \ref{prop_orthogonal_sum} will be called a \textbf{bounded orthogonal sum}.
\end{definition}

For this section, we finally describe a space of measures $\mc{M}_X$ which is suitable for dealing with transfunctions later.

\subsection*{Ample Spaces}

\begin{definition}\label{def_ample}
Let $(X,\tau)$ be a topological space. A space of finite positive measures or of vector measures $\mc{M}_X$ on $(X,\Sigma_X)$ is called \textbf{ample} if the following statements hold:
\begin{enumerate}[label=(\roman*)]
\item $\mc{M}_X$ is closed under projections;
\item $\mc{M}_X$ is closed under bounded orthogonal summations;
\item Every nonempty open set in $X$ carries some nonzero measure in $\mc{M}_X$.
\end{enumerate} 
\end{definition}

\begin{example}\label{ex_ample_Lebesgue}
	Let $\lambda$ be the Lebesgue measure on $(\mb{R},\mc{B}(\mb{R}))$, and let $\mc{M}_\lambda$ be the space of all measures that are projections of $\lambda$ onto Borel sets in $\mc{B}(\mb{R})$, that is, $\mc{M}_\lambda := \{\pi_A \lambda : A \in \mc{B}(\mb{R})\}$. Then it is easy to verify by observation and by Proposition 1.11 that $\mc{M}_\lambda$ is closed under projections and closed under bounded orthogonal summations. Additionally, $\lambda$ is strictly-positive, so it follows that $\mc{M}_\lambda$ is ample. Since $\lambda$ is $\sigma$-finite, the subset $\mc{M}_\lambda^F$ of finite measures from $\mc{M}_\lambda$ is also ample. By the argument of Proposition 1.9, it follows that both $\mc{M}_\lambda$ and $\mc{M}_\lambda^F$ only contain non-atomic measures.
\end{example}

\begin{example}\label{ex_ample_abs_cont}
Generalizing from Example \ref{ex_ample_Lebesgue}, let $\lambda$ be a finite strictly-positive measure on a topological measurable space $(X,\Sigma_X)$, and let $\mc{M}_\lambda := \{ \pi_A \lambda : A \in \Sigma_X \}$. By the same reasons as from Example 1.10, $\mc{M}_\lambda$ is an ample space of finite measures.
\end{example}

%
%

Ample spaces will be useful for transfunctions because we will want to decompose a domain measure into bounded orthogonal sums of projections and use local properties to determine the behavior of each projection. If the transfunctions are of a certain type, then summing the outputs will result in the output of the original measure.

As Examples \ref{ex_ample_Lebesgue} and \ref{ex_ample_abs_cont} illustrate, having finite strictly-positive measures on a space $(X,\Sigma_X)$ can be used to easily generate an ample space of measures. Such measures are guaranteed on second-countable locally compact Hausdorff spaces (see \cite{Bentley}) and on compact groups (any left or right Haar measure).



\section{Transfunctions}

From this point forward, we will assume that $X$ and $Y$ are second-countable topological spaces, that $\Sigma_X$ and $\Sigma_Y$ are collections of Borel subsets of $X$ and $Y$, respectively, and that any transfunction $\Phi: \mc{M}_X \to \mc{M}_Y$ will be defined on an ample space $\mc{M}_X$ unless otherwise specified. This will allow us to work with various properties regarding the input measures and the transfunction.

\begin{example}\label{ex_push_forward}
Let $f: X \to Y$ be a measurable function. Then the push forward operator
$f_\# : \mc{M}_X \to \mc{M}_Y $ defined by $f_\# (\mu) (B) := \mu (f^{-1}(B))$ is a transfunction. We identify the measurable function $f$ with the corresponding push-forward operator $f_\#$. 
\end{example}

\subsection*{Strongly $\sigma$-Additive and Weakly $\sigma$-Additive Transfunctions}

\begin{definition}\label{def_sigma_additive}
Let $\Phi: X \rsq Y$ be a transfunction with $\mc{M}_X$ closed under bounded orthogonal sums. We say that \textbf{$\Phi$ is weakly $\sigma$-additive} if every bounded orthogonal sum $\sum_{i=1}^\infty \mu_i$ in $\mc{M}_X$ merits that $\Phi(\sum_{i=1}^\infty \mu_i) = \sum_{i=1}^\infty \Phi \mu_i$.
As a consequence, $\Phi$ is also \textbf{weakly monotone}, which means that $\Phi\mu \le \Phi(\mu + \nu)$ for each orthogonal pair of measures $\mu$ and $\nu$. 
If $\mc{M}_X$ is closed under bounded sums and the equality persists when orthogonality is relaxed, then we say that $\Phi$ is \textbf{strongly $\sigma$-additive} and, by consequence, that $\Phi$ is \textbf{strongly monotone}. 
\end{definition}

Notice that $\sigma$-strong additivity implies $\sigma$-weak additivity.

\begin{example}\label{ex_measurable_is_sigma_week}
A measurable function $f_\#: X \rsq Y$ with $\mc{M}_X$ closed under bounded orthogonal sums is strongly $\sigma$-additive: if $\sum_{i=1}^\infty \mu_i$ is a bounded sum and $B \in \Sigma_Y$, then $f_\# \left(\sum_{i=1}^\infty \mu_i\right) (B) = \left(\sum_{i=1}^\infty \mu_i \right) (f^{-1} (B)) = \sum_{i=1}^\infty \mu_i ( f^{-1} (B)) = \sum_{i=1}^\infty f_\# \mu_i (B)$, so $f_\# \sum_{i=1}^\infty \mu_i = \sum_{i=1}^\infty f_\# \mu_i$. It follows that $f_\#$ is weakly $\sigma$-additive, which is sufficient for most of the theory in later sections.
\end{example}

In all of our important examples and theorems, we shall be dealing with weakly $\sigma$-additive transfunctions. 
This property is useful for models because such a transfunction can delegate its action among orthogonal projections and accumulate their individual actions afterwards. In transportation problems, one can view a weakly $\sigma$-additive transfunction as a transportation method which can be subdivided between different districts and locally performed by each region independent of each other. Thus, a leading supervisor on a national scale can delegate their major transportation plan among their employees in different states, then those employess can further delegate their plans among \emph{their} employees in different counties, and so on.

\begin{remark}\label{rem_spatial}In future sections, we shall often discuss the spatial relationship between a measure $\mu$ and its output $\Phi \mu$. If $A \subseteq X$ and $B \subseteq Y$ are measurable sets such that $\mu \sqsubset A$ implies $\Phi \mu \sqsubset B$, we shall notate this spatial relationship via ``$\Phi(A) \sqsubset B$". The following proposition, and the subsequent corollary, are frequently utilized facts for spatial relationships in future proofs.
\end{remark}

{\begin{proposition}\label{prop_spatial_relationship}
	Let $\Phi: X \rsq Y$ be a weakly $\sigma$-additive transfunction. Let $A,A'$ and the sequence $(A_i)_{i=1}^\infty$ be from $\Sigma_X$ and let $B, B'$ and the sequence $(B_j)_{j=1}^\infty$ be from $\Sigma_Y$.
\begin{enumerate}
\item[\rm{(i)}] If $\Phi(A) \sqsubset B$, if $A' \subseteq A$ and $B' \supseteq B$, then $\Phi(A') \sqsubset B'$;
\item[\rm{(ii)}]  If $\Phi(A_i) \sqsubset B_j$ for all $i,j \in \mb{N}$, then $\Phi(\cup_{i=1}^\infty A_i) \sqsubset \cap_{j=1}^\infty B_j$.
\item[\rm{(iii)}]  If $\Phi(A_i) \sqsubset B_i$ for all $i \in \mb{N}$, then $\Phi(\cap_{i=1}^\infty A_i) \sqsubset \cap_{j=1}^\infty B_j$ and $\Phi(\cup_{i=1}^\infty A_i) \sqsubset \cup_{j=1}^\infty B_j$.
\end{enumerate}
\end{proposition}

\begin{proof} 
	(i): Let $\mu \sqsubset A'$. Using Proposition \ref{prop_carrier} (i) twice and the definition of $\Phi(A) \sqsubset B$, we have that $\mu \sqsubset A$, so $\Phi\mu \sqsubset B$ and $\Phi\mu \sqsubset B'$. Therefore (i) is shown. 
	
	(ii): Let $\mu \in \mc{M}_X$ such that $\mu \sqsubset \cup_{i=1}^\infty A_i$. Then define the sequence $(T_i)_{i=1}^\infty$ by $T_i = A_i - \cup_{n<i} A_n$. Note that $T_i \subseteq A_i$ for all natural $i$ and that $\{T_i\}_{i=1}^\infty$ is a pairwise disjoint set sequence with $\uplus_{i=1}^\infty T_i = \cup_{i=1}^\infty A_i$. Fix $i,j \in \mb{N}$. It follows from part (i) that $\Phi(T_i) \sqsubset B_j$. Since $\pi_{T_i} \mu \sqsubset T_i$, we obtain $\Phi(\pi_{T_i} \mu) \sqsubset B_j$. Since $j$ was arbitrary, by Proposition \ref{prop_carrier} (ii) we have that $\Phi(\pi_{T_i} \mu) \sqsubset \cap_{i=1}^\infty B_j$. Since $i$ was arbitrary, by $\sigma$-weak additivity of $\Phi$ and by Proposition \ref{prop_carrier} (iv) it follows that 
	$$
	\Phi(\mu) = \Phi\left(\sum_{i=1}^\infty \pi_{T_i} \mu\right) = \sum_{i=1}^\infty \Phi(\pi_{T_i} \mu) \sqsubset \bigcap_{i=1}^\infty B .
	$$
	
	(iii): $\Phi(\cap_{i=1}^\infty A_i) \sqsubset B_i$ and $\Phi(A_i) \sqsubset \cup_{i=1}^\infty B_i$ for all $i \in \mb{N}$ by part (i), then by part (ii) we have that $\Phi(\cap_{i=1}^\infty A_i) \sqsubset \cap_{j=1}^\infty B_j$ and $\Phi(\cup_{i=1}^\infty A_i) \sqsubset \cup_{j=1}^\infty B_j$.
\end{proof}


The following will be useful in the characterization of continuous functions:

\begin{corollary}\label{cor_spatial_relationship}
	Let $\Phi: X \rsq Y$ be a weakly $\sigma$-additive transfunction. Let $U$ be open in $X$ with open cover $\{S_i : i \in I\}$, and let $B$ be measurable in $Y$. Then $\Phi(S_i)\sqsubset B$ for all $i \in I$ implies that $\Phi(U)\sqsubset B$. In particular, if $\mu$ is a measure on $X$ and if $\Phi(\pi_{S_i} \mu) \sqsubset B$ for all $i \in I$, then $\Phi(\pi_U \mu) \sqsubset B$.
\end{corollary}
\begin{proof}
	By second-countability of $X$, we have that $U$ has a countable subcover $\{S_{i(n)} : n \in \mb{N}\}$, then apply Proposition \ref{prop_spatial_relationship} (ii) by setting $A_i = S_{i(n)}$ and $B_j = B$ to show the first claim. For the second claim, note that the partition $(T_i)_{i=1}^\infty$ from the proof of Proposition \ref{prop_spatial_relationship} (ii) satisfies $\Phi(\pi_{T_i} \mu) \le \Phi(\pi_{A_i} \mu)$ for all $i \in \mb{N}$ by weak monotonicity of $\Phi$. Then $\Phi(\pi_U \mu) = \sum_{i=1}^\infty \Phi(\pi_{T_i} \mu) \le \sum_{i=1}^\infty \Phi(\pi_{A_i} \mu) \sqsubset B$ by Proposition \ref{prop_carrier} (iv). So $\Phi(\pi_U \mu) \sqsubset B$ by Proposition \ref{prop_carrier} (iii).
\end{proof}

Since $\mb{R}^N$ and $\mb{C}^N$ are important examples of second-countable spaces with small open sets, Corollary \ref{cor_spatial_relationship} ensures that $\sigma$-weak additivity is enough to guarantee that transfunctions behave only on a local scale; that this, there are no open sets `large enough' to have positive measure without some local `small' open subset also having positive measure.

\subsection*{Support of a Transfunction}

\begin{definition}\label{def_vanish}
A transfunction $\Phi : X \rsq Y$ is said to \textbf{vanish on open set $U$} if $\Phi(U)\sqsubset \varnothing$. Let $\mc{V}_{\Phi}$ denote the collection of vanishing sets of $\Phi$.
\end{definition}
In a general setting, this statement may be vacuously true. However, in our setting, $\Phi$ is defined on an ample space of measures, so the definition is meaningful. Is the union of vanishing sets $\cup\mc{V}_{\Phi}$ also a vanishing set? In our setting, it is true for weakly $\sigma$-additive transfunctions via Corollary \ref{cor_spatial_relationship} because $\cup \mc{V}_\Phi$ has a countable subcover.

%

\begin{definition}\label{def_null_space_spatial_support}
Let $\Phi : X \rsq Y$ be a weakly $\sigma$-additive. Then $\cup \mc{V}_\Phi$ is called the \textbf{null space of $\Phi$}, notated as $\nulls\Phi$. Its complement, $(\cup \mc{V}_\Phi)^c$, is called the \textbf{spatial support of $\Phi$}, notated as $\supp\Phi$.
\end{definition}

Observe that 
\begin{align*}
\Phi \mu = \Phi \left(\pi_{\displaystyle \nulls\Phi}\mu + \pi_{\displaystyle \supp\Phi} \mu\right) = \Phi \left(\pi_{\displaystyle \nulls\Phi}\mu\right) + \Phi\left(\pi_{\displaystyle \supp\Phi} \mu \right) = \Phi\left(\pi_{\displaystyle \supp\Phi} \mu \right),
\end{align*}
which implies that $\Phi$ is essentially a transfunction between the subspace $\supp\Phi$ and $Y$, that is, $\Phi : \supp\Phi \rsq Y$. This realization justifies calling $\supp\Phi$ the spatial support of $\Phi$.\\

\begin{definition}\label{def_non-vanishing}
A transfunction $\Phi : X \rsq Y$ is called \textbf{non-vanishing} if $\Phi$ has no non-empty vanishing sets, that is, if $\supp\Phi = X$. Furthermore, $\Phi$ is \textbf{norm-preserving} if $||\Phi \mu|| = ||\mu||$ for all $\mu$ on $X$.
\end{definition}

The ample and non-vanishing characteristics are important for the main theorems later. In transportation problems (Monge-Kantorovich), we will usually only consider transfunctions that are norm-preserving. However, in other applications like population dynamics, it is reasonable to abandon norm-preservation and even non-vanishment.

\section{Localized Transfunctions}
\subsection*{Definitions and Examples}
At this point, we make a further assumption that $X$ and $Y$ are metric spaces. For each of the scenarios below, let $\Phi: X \rsq Y$ denote a transfunction, and let $f: X \to Y$ either denote a function without structure or a measurable function depending on the context.

We would like to use the devices from continuity of functions to help develop a theory on transfunctions that behave in a local manner. In particular, we use the version of continuity of functions between metric spaces which states that $f$ is continuous at $x$ if and only if
\begin{align*}
\forall \eps>0  ~ \exists \delta > 0 : A \subseteq B(x;\delta) \implies f(A) \subseteq B(y;\eps) ,
\end{align*}
or, when written in a simpler way, that
\begin{align*}
\forall \eps>0 ~  \exists \delta > 0 :  f(B(x;\delta)) \subseteq B(f(x);\eps) .
\end{align*}
We shall adapt this statement to ``localized" transfunctions by replacing subsets with carried measures: that is, we will replace ``$A \subseteq B(x;\delta)$" with ``$\mu \sqsubset B(x;\delta)$ and ``$f(A) \subseteq B(y;\eps)$" with ``$\Phi(\mu) \sqsubset B(y;\eps)$".

\begin{definition}\label{def_localized}
Let $x\in X$. Then we say that a transfunction \textbf{$\Phi$ is localized at $x$} if there exists $\delta, \eps > 0$ and $y \in Y$ such that 
$$\Phi(B(x,\delta )) \sqsubset B(y,\eps ).$$  
When $\eps$ needs emphasis, we say that \textbf{$\Phi$ is $\eps$-localized at $x$}. If $y$ also needs emphasis, we say that \textbf{$\Phi$ is $\eps$-localized at $(x,y)$}. We rarely need to emphasize $\delta$, but when necessary, we say that \textbf{$\Phi$ is $(\delta,\eps)$-localized at $(x,y)$}. We will always use one of the four phrases above when stating localization of $\Phi$. 

We will say that \textbf{$\Phi$ is $0$-localized at $x$} if $\Phi$ is $\eps$-localized at $x$ for all $ \eps > 0$. Although this definition has the same notation as  the previous definition, there will be no confusion. $0$-localized transfunctions have a special role in classifying continuous functions. Note in the definition of $0$-localization that the values for $\delta$ and $y$ may depend on $\eps$.
\end{definition}

\begin{definition}\label{def_localized_on_set}
If $\Phi$ is localized at every point in some set $A \in \Sigma_X$, then $\Phi$ is said to be \textbf{localized on $A$}. We define a function $E := E_\Phi: X \to [0,\infty]$ via 
$$
E(x) := \inf \{\eps : \Phi \text{ is } \eps\text{-localized at } x\}
$$
which measures how localized $\Phi$ can be throughout $X$, and we say that \textbf{$\Phi$ is $E^+$-localized}, which means that $\Phi$ is $(E(x)+\eta)$-localized at $x$ for all $\eta > 0$ whenever $E(x) < \infty$, and that $\Phi$ is not localized at $x$ when $E(x) = \infty$.
\end{definition}

\begin{definition}\label{def_localized_via_function}
Let $A \subseteq X$ and let $f: X \to Y$ be function. We say that \textbf{$\Phi$ is $\eps$-localized on A via $f$} or that \textbf{$\Phi$ is $\eps$-close to $f$ on $A$} if $\Phi$ is $\eps$-localized at $(x,f(x))$ for all $x\in A$. A similar definition is made by replacing $\eps$ with $E^+$ or with $0$.
\end{definition}

\begin{definition}\label{def_uniformly_localized}
Let $A \subseteq X$. We say that \textbf{$\Phi$ is uniformly localized on $A$} if there exist $\eps>0$ and $\delta >0$ such that $\Phi$ is $(\delta,\eps)$-localized on $A$. If $\delta$ and $\eps$ are to be emphasized, then we say that \textbf{$\Phi$ is uniformly $(\delta,\eps)$-localized on $A$}. If only $\eps$ is to be emphasized, then we say that \textbf{$\Phi$ is uniformly $\eps$-localized on $A$}.
\end{definition}

It is worth noting that transfunctions are not necessarily localized anywhere, and the list of definitions above give rise to a lot of examples of transfunctions which will appear below; many of these examples are made clear by the following short proposition.\\

\begin{proposition}\label{prop_commute_carrier_with_inverse_function}
Let $f$ be a measurable function, let $\mu \in \mc{M}_X$ be a positive measure or a vector measure and let $A \in \Sigma_X$. Then $|f_\# \mu| \le f_\#|\mu|$. If $\mu$ is a positive measure, then $\mu \sqsubset f^{-1}(A)$ if and only if $f_\# \mu \sqsubset A$, and if $\mu$ is a vector measure, the forward implication holds.
\end{proposition}

\begin{proof}
To show the first claim, notice that for all $B \in \Sigma_X$,
\begin{align*}
|f_\# \mu| (B) &= \sup \left\{ \textstyle \sum_{i=1}^n ||f_\# \mu (B_i)|| : \uplus_{i=1}^n B_i = B \right\} \\ 
&= \sup \left\{ \textstyle \sum_{i=1}^n ||\mu (f^{-1} (B_i))|| : \uplus_{i=1}^n f^{-1} (B_i) = f^{-1} (B) \right\} \\ 
&\le |\mu| (f^{-1} B) = f_\# |\mu| (B).
\end{align*}

The second claim then follows for $\mu$ positive from
$$
\mu \sqsubset f^{-1}(A) \Leftrightarrow 0 = \mu((f^{-1} (A))^c) = \mu(f^{-1}(A^c)) = \mu \circ f^{-1} (A^c) = 0 \Leftrightarrow f_\# \mu \sqsubset A,
$$
and when $\mu$ is a vector measure, the forward statement follows from the first claim with Proposition \ref{prop_carrier} (iii).
\end{proof}

\begin{example}\label{ex_localized_continuous}
Let $f: X \to Y$ be continuous. Then $\Phi := f_\#$ will have $E = \mf{0}$. To show this, let $x \in X$ and let $\eps > 0$. Then there exists $\delta>0$ such that $B(x; \delta) \subseteq f^{-1} B(f(x); \eps)$. Therefore, $\mu \sqsubset B(x;\delta) \subseteq f^{-1} B(f(x);\eps)$ implies from Proposition \ref{prop_commute_carrier_with_inverse_function} that $\Phi \mu = f_\# \mu \sqsubset B(f(x);\eps)$. Therefore, $E(x) \le \eps$ for every $\eps > 0$ and $x \in X$. Hence, $E(x) = 0$ for all $x \in X$, showing that $E= \mf{0}$. 
\end{example}

\begin{example}\label{ex_localized_Heaviside}
Let $H:= H_0: \mb{R} \to \mb{R}$ be the Heaviside function centered at $0$. Then define $\Phi := H_\#$. Since $H$ is continuous everywhere except at $0$, it follows that $E(x) = 0$ for all $x \not= 0$ by Example \ref{ex_localized_continuous}. However, consider when $x = 0$. Let $\delta > 0$. Since $H^{-1}(0) = (-\infty,0)$ and $H^{-1}(1) = [0, \infty)$, and since $B(0;\delta) \not\subseteq (-\infty,0)$ and $B(0;\delta) \not\subseteq [0,\infty)$, it follows that $B(0;\delta) \subseteq H^{-1} A$ only when $\{0,1\} \subseteq A$. As a result, $H_\#(B(0;\delta))\sqsubset A$ only when $\{0,1\} \subseteq A$. In order for $A$ to be an $\eps$-ball containing $0$ and $1$, it must follow that $\eps > 1/2$. Therefore, $E(0) = 1/2$. Hence, $E(x) = \frac12 \delta_0(x)$ for all $x\in \mb{R}$. 
\end{example}

\begin{example}\label{ex_localized_sum_of_Heavisides}
Consider the measurable function $g: \mb{R} \to \mb{R}$ via $g = \sum_{n=0}^\infty 2^n H_n$, and define $\Phi = g_\#$. Then $E(n) = 2^{n-1}$ for each $n \in \mb{N}$ and $E(x) = 0$ for $x \in \mb{R} - \mb{N}$, meaning that $g_\#$ is localized on $\mb{R}$ but that $\sup_{x \in \mb{R}} E(x) = \infty$. 
\end{example}

\begin{example}\label{ex_projection_localized}
Let $A \in \Sigma_X$. Then the projection transfunction $\Phi: \mu \mapsto \pi_A \mu$ is $0$-localized via the identity function $i: X \to X$ since every carrier of $\mu$ is also a carrier of $\pi_A \mu$. $\Phi$ is also strongly $\sigma$-additive.

\end{example}

\begin{example}\label{ex_localized_nowhere}
Let $Y = \mb{R}$ and $\Sigma_Y = \mc{B}(\mb{R})$. Let $\{q_n\}_{n=1}^\infty$ be an enumeration of $\mb{Q}$ and let $\nu \sqsubset \mb{Q}$ be a measure defined via $\nu (A) = \sum_{q_n \in A} 2^{-n}$. Then define $\Phi: \mu \mapsto ||\mu|| \nu$. Then $\Phi \mu \sqsubset A$ implies that $A \supseteq \mb{Q}$ whenever $\mu$ is non-zero. Since $\mb{Q}$ cannot be contained in any $\eps$-ball, it follows that $\Phi$ is not localized anywhere.
\end{example}

\begin{example}\label{ex_localized_convolution}
	Let $X=Y=\mb{R}^d$ and let $\lambda^d$ be the Lebesgue measure on $\mb{R}^d$. For some $\eps > 0$, define $\kappa := \pi_{B(0;\eps)}\lambda^d$. Then $\Phi: \mu \mapsto \mu \ast \kappa$ is a strongly $\sigma$-additive transfunction which is $\eps$-localized on $X$ via the identity map $i: X \to X$ and is norm-preserving when $||\kappa||=1$. One sees that $E=\eps$ on $X$ and that $\kappa$ scales and smoothes or diffuses $\mu$ to obtain $\Phi \mu$. $\Phi$ is not identified with any function; hence, $\Phi$ does not resemble any traditional transport mapping. 
\end{example}

\begin{example}\label{ex_localized_convolution_and_function}
	Continuing Example \ref{ex_localized_convolution}, let us additionally consider a continuous function $f:X \to X$. Then the transfunctions $\Psi_1 \mu \mapsto (f_\# \mu) \ast \kappa$ and $\Psi_2 \mu \mapsto f_\# (\mu \ast \kappa)$ are strongly $\sigma$-additive and localized on $X$. $E_{\Psi_1} = \eps$ on $X$ but $E_{\Psi_2}$ is not easy to calculate as it depends on $f$. One could also consider $f$ to be measurable with $\mu \mapsto f_\# \mu$ being $E^+$-localized. Then $\Psi_1$ is $(E+\eps)^+$-localized and maintains its other properties. 

\end{example}

If $U$ denotes the set of points in $X$ where $\Phi$ is localized, then the previous examples show that $E|_U: U \to [0,\infty)$ does not have to be continuous. However, as the next proposition states, it does have to be upper-semi-continuous on $U$, implying that $U$ is an open set. \\

\begin{proposition}\label{prop_upper_semi_continuous}
Let $U$ and $E$ be defined as above. Then $E|_U$ is an upper-semi continuous function and $U$ is an open set.
\end{proposition}

\begin{proof}
Let $x \in U$ with $E(x) = \eta$. Let $\eps > \eta$. Then there exists some $y \in Y$ and some $\delta>0$ such that $\Phi (B(x;\delta)) \sqsubset B(y;\eps)$. Choose some arbitrary $x_1 \in B(x;\delta)$ different from $x$. Then there exists some $\delta_1$ with $B(x_1;\delta_1) \subset B(x;\delta)$ so that $\Phi (B(x_1;\delta_1)) \sqsubset B(y;\eps)$. Therefore, $\Phi$ is $\eps$-localized at $(x_1,y)$, yielding that $E(x_1) \le \eps$ and that $x_1 \in U$. Since $x_1 \in B(x;\delta)$ was arbitrary, this means that $\sup E(B(x;\delta)) \le \eps$ and that $B(x;\delta) \subseteq U$. Since $\eps$ was arbitrary, this means that $\displaystyle \limsup_{\delta \to 0} E(B(x;\delta)) \le \eta = E(x)$, meaning that $E$ is upper-semi-continuous. Since $x$ was arbitrary, this means that $U$ is open.
\end{proof}

\section{$0$-Localized Transfunctions}

When $f: X \to Y$ is continuous, we know that $f_\#$ is a well-behaved transfunction that is weakly $\sigma$-additive and norm-preserving. We also know that $f_\#$ is $0$-localized on $X$ by Example \ref{ex_localized_continuous}. It turns out that these three properties will characterize continuous functions in our setting, leading to Theorem \ref{thm_0-localized_continuous_char}. First, we will need the following proposition to find the continuous function:

\begin{proposition}\label{prop_0-localized_continuous_func}
Let $X$ be a metric space with an ample family of measures $\mc{M}_X$ and let $Y$ be a complete metric space. For any $A \in \Sigma_X$ and for any non-vanishing transfunction $\Phi:X \rsq Y$ which is $0$-localized on $A$ there is a unique continuous function $f: A \to Y$ such that $\Phi$ is $0$-close to $f$ on $A$.
\end{proposition}

\begin{proof}
Since $\Phi$ is $0$-localized on $A$, it follows that $E(x) = 0$ for all $x\in A$. Then for any fixed $x \in A$, there are $\delta_n>0$ and $y_n\in Y$, $n\in\mb{N}$, such that $\Phi(B(x, \delta_n))\sqsubset B(y_n,\frac1{n})$ for every $n\in\mb{N}$.

First we show that $d(y_m, y_n) \leq \frac1{m} + \frac1{n}$ for all $m,n\in\mb{N}$. Suppose
$d(y_m, y_n) > \frac1{m} + \frac1{n}$ for some $m,n\in\mb{N}$. Then $B(y_m; \frac1{m}) \cap B(y_n; \frac1{n}) = \varnothing$. Since $\mc{M}_X$ is ample, there is a non-zero measure $\nu\sqsubset B(x;\delta_m) \cap B(x;\delta_n)$. But then $\Phi(\nu)\sqsubset B(y_m; \frac1{m}) \cap B(y_n; \frac1{n})= \varnothing$, which is impossible since $\Phi$ is non-vanishing. 

Since $d(y_m, y_n) \le \frac1{m} + \frac1{n}$ for all $m,n\in\mb{N}$,  $(y_n)$ is a Cauchy sequence in the complete metric space $Y$. So there exists $y \in Y$ with $y_n \to y$. Furthermore, $B(y_n ,\frac1{n}) \subseteq B(y , \frac1{n}+d(y_n,y))$ for $n\in\mb{N}$. Indeed, $y$ is the unique point in $Y$ with this property and $\Phi$ is $0$-localized at $(x,y)$. Now we define $f: A \rightarrow Y$ by $f(x) = y$, where $y\in Y$ is the unique point such that $\Phi$ is $0$-localized at $(x,y)$. Clearly $\Phi$ is $0$-localized on $A$ via $f$. 

We now show that $f$ is continuous on A. Let $x_n \to x_0$ in $A$. Define $y_0 = f(x_0)$ and $y_n = f(x_n)$ for $n\in\mb{N}$. Let $\eps > 0$. Then there is a $\delta>0$ such that 
$\Phi(B(x_0,\delta)) \sqsubset B(y_0,\eps)$. Let $N\in\mb{N}$ be such that $d(x_m,x_0) < \delta/2$ for $m\ge N$. For every $m \ge N$ there is a $\delta_m < \delta/2$ such that $\Phi(B(x_m, \delta_m)) \sqsubset B(y_m, \eps)$. Then $B(x_m,\delta_m) \subset B(x_0,\delta)$ implies that $\Phi(B(x_m,\delta_m)) \sqsubset B(y_0,\eps) \cap B(y_m,\eps)$. Consequently, $d(y_m,y_0) \le 2\eps$, since $\mc{M}_X$ is ample and $\Phi$ is non-vanishing.
\end{proof}

\begin{theorem}\label{thm_0-localized_continuous_char}
Let $X$ be a metric space with an ample family of measures $\mc{M}_X$, $Y$ a complete metric space, and $\Phi: X \rsq Y$ a non-vanishing transfunction. Then $\Phi=f_\#$ for some continuous function $f:X\to Y$ if and only if $\Phi$ is norm-preserving, weakly $\sigma$-additive, and $0$-localized on $X$.
\end{theorem}

\begin{proof}
We only need to show that if $\Phi$ is norm-preserving, weakly $\sigma$-additive, and $0$-localized on $X$, then $\Phi=f_\#$ for some continuous function $f:X\to Y$. Let $f: X \to Y$ be the unique continuous function guaranteed by Proposition \ref{prop_0-localized_continuous_func}. We will show that $\Phi=f_\#$. 

Let $V\subseteq Y$ be an open set.  Define $U = f^{-1} (V)$. By Proposition \ref{prop_0-localized_continuous_func}, there exists an open ball cover $\{B(x; \delta_x) : x \in U\}$ of $U$ such that $\Phi(B(x; \delta_x)) \sqsubset V$ for all $x \in X$. By Corollary \ref{cor_spatial_relationship}, we have that $\Phi(U) \sqsubset V$. 

Similarly, if we define $W= f^{-1} \left(\left(\overline{V}\right)^c\right)$, which is also open, we have that $\Phi(W) \sqsubset \left(\overline{V}\right)^c$. 

Next, we define $Z := f^{-1}(\partial V) = f^{-1} (\overline{V} \cap V^c)$, which is closed, and for each $n \in \mb{N}$, let $L_n = \cup_{y\in \partial V} B(y; 1/n)$, which is an open set. Since $\Phi$ is $0$-localized, for each $x \in Z$ and $n\in\mb{N}$, there exists $\delta_{x,n}>0$ such that $\Phi(B(x; \delta_{x,n})) \sqsubset L_n$. For $K_n = \cup_{x\in Z} B(x; \delta_{x,n})$ we have $\Phi(K_n) \sqsubset L_n$ for all $n\in\mb{N}$, by Corollary \ref{cor_spatial_relationship}. Noting that $\cap_{n=1}^\infty K_n =Z$ and $\cap_{n=1}^\infty L_n = \partial V$, it is clear by Proposition \ref{prop_spatial_relationship} (iii) that $\Phi(Z) \sqsubset \partial V$. 

Let $\mu \in \mc{M}_X$. Since $\pi_U \mu \sqsubset U$ and $\Phi(U) \sqsubset V$, we have that $\Phi(\pi_U \mu) \sqsubset V$ and we have via norm-preservation of $\Phi$ that
$$
\Phi(\pi_U \mu) (V) = ||\Phi(\pi_U \mu)|| = ||\pi_U \mu|| = \mu(U) = f_\# (\mu) (V).
$$

Since $\Phi(W)\sqsubset \left(\overline{V}\right)^c$ and $V \cap \left(\overline{V}\right)^c = \varnothing$, it follows that $\Phi(\pi_W \mu) (V) = 0.$ Again, since $\Phi(Z)\sqsubset \partial V$ and $V \cap \partial V = \varnothing$, it follows that $\Phi(\pi_Z \mu) (V) = 0$.

From the above we obtain 
$$
\Phi(\mu)(V) = \Phi(\pi_U \mu)(V) + \Phi(\pi_W \mu)(V) + \Phi(\pi_Z \mu)(V) = f_\# (\mu) (V).
$$ 
Moreover, since $\Phi(\mu)$ and $f_\#(\mu)$ are finite measures which agree on open sets, they must agree on all sets in $\Sigma_Y$ by the $\pi-\lambda$ Theorem.  Finally, since $\mu\in\mc{M}_X$ is arbitrary, we have $\Phi = f_\#$.
\end{proof}

\begin{theorem}
If norm-preservation of $\Phi$ is omitted from the hypothesis of Theorem \ref{thm_0-localized_continuous_char} and if $f$ is the continuous function found in Proposition \ref{prop_0-localized_continuous_func}, then for every $\mu \in \mc{M}_X$, it follows that $\Phi \mu \ll f_\# \mu$, that is, $\Phi \mu$ is absolutely continuous with respect to $f_\# \mu$ for all $\mu \in \mc{M}_X$.
\end{theorem}

\begin{proof}
Let $\mu \in \mc{M}_X$. To show that $\Phi \mu \ll f_\# \mu$, it is equivalent to show that for all $B \in \Sigma_Y$, $f_\# \mu \sqsubset B$ implies that $\Phi \mu \sqsubset B$, or that $\Phi( f^{-1}(B)) \sqsubset B$ by Proposition \ref{prop_commute_carrier_with_inverse_function}. 

Let $\mc{H} := \{B \in \Sigma_Y : \Phi(f^{-1}(B)) \sqsubset B\}$. 

Let $V$ and $U := f^{-1}(V)$ be open sets. If $\mu \sqsubset U$, then $\pi_U \mu = \mu$ (Corollary \ref{cor_projection}) and $\Phi \mu \sqsubset V$ ( proof of Theorem \ref{thm_0-localized_continuous_char}). This means that $\Phi (f^{-1}(V)) \sqsubset V$ and that $\mc{H}$ contains $V$. Therefore $\mc{H}$ contains all open sets in $Y$. Since $Y$ is a metric space, it follows that all Borel subsets of $Y$ can be constructed from the open subsets inductively by considering countable intersections and countable unions of previously constructed sets. In other words, if we define the family $\{\mc{G}_\alpha : \alpha \text{ is a countable ordinal}\}$ of collections of subsets in the following way, each Borel set belongs to one of these collections:
\begin{itemize}
\item[] $\mc{G}_0 := \tau_Y = \{V : V \text{ is open in Y} \}$ ;
\item[] for non-preceded ordinal $\alpha$ and even $n \in \mb{N}$,  $\mc{G}_{\alpha +n+1} := \{ \cap_{i=1}^\infty A_i : A_i \in \mc{G}_{\alpha +n} \}$ ;
\item[] for non-preceded ordinal $\alpha$ and odd $n \in \mb{N}$,  $\mc{G}_{\alpha +n+1} := \{ \cup_{i=1}^\infty A_i : A_i \in \mc{G}_{\alpha +n} \}$ ; 
\item[] for limit ordinal $\alpha$, $\mc{G}_{\alpha} := \{ \cup_{i=1}^\infty A_i : A_i \in \cup_{\beta < \alpha} \mc{G}_\beta \}$.
\end{itemize}

To complete the proof via transfinite induction, we need to show that the collection $\mc{H}$ is closed under countable intersections and countable unions. Let $(B_i)_{i=1}^\infty$ be a sequence in $\mc{H}$. For all natural $i$, $\Phi(f^{-1} (B_i)) \sqsubset B_i$. Therefore, by Proposition \ref{prop_spatial_relationship} (iii), we have that $\Phi(f^{-1} (\cap_{i=1}^\infty B_i)) \sqsubset \cap_{i=1}^\infty B_i$ and $\Phi(f^{-1} (\cup_{i=1}^\infty B_i)) \sqsubset \cup_{i=1}^\infty B_i$. Hence,  $\cap_{i=1}^\infty B_i \in \mc{H}$  and $\cup_{i=1}^\infty B_i \in \mc{H}$.
\end{proof}

This allows every $\mu$ to correspond to some nonnegative $f_\# \mu$-integrable $g_\mu : Y \to [0,\infty]$ such that $\displaystyle \Phi \mu (B) = \int_B g_\mu d(f_\# \mu) = \int_{f^{-1} (B)} g_\mu \circ f d\mu$ for all $\mu$. If $g_\mu$ does not depend on $\mu$, then $\Phi$ is strongly additive.

We can also characterize transfunctions which are identically measurable functions, but under stricter settings. First, we discuss restrictions of transfunctions.

\begin{definition}\label{def_transfunction_restriction}
Let $\Phi : X \rsq Y$ be a transfunction, and let $A\subseteq X$ be measurable. Then  $\Phi \circ \pi_A$, or $\mu \mapsto \Phi(\pi_A \mu)$ is called the \textbf{restriction of $\Phi$ to $A$}. 
\end{definition} 


Note that $\Phi\circ\pi_A= \Phi$ when $\supp\Phi \subseteq A$ and that $\Phi\circ\pi_B = 0$ when $B \subseteq \nulls\Phi$.

\begin{theorem}\label{thm_transfunction_measurable_char}
Let $X$ be locally compact, let $\lambda$ be a finite regular measure on $X$, and let $\mc{M}_X$ only have measures absolutely continuous with respect to $\lambda$. Let $\Phi: X \rsq Y$ be a weakly $\sigma$-additive transfunction. Then $\Phi$ is a measurable function if and only if there exists a sequence of compact sets $\{F_n\}_{n=1}^\infty$ such that $\lambda(F_n^c) < \displaystyle \frac{1}{n}$ and that $\Phi \circ \pi_{F_n}$ is identified with some continuous function on $F_n$.
\end{theorem}

\begin{proof}
The forward direction is a straight-forward consequence of Lusin's Theorem, where the measurable and continuous functions are identified with the respective transfunctions. 

We now prove the reverse direction. For each natural $n$, let $\Phi \circ \pi_{F_n}$ be identified with continuous function $f_n : F_n \to Y$. Let $i \not= j$. If $\lambda(F_i \cap F_j) > 0$, then by Lemma 7.5.2 from Cohn \cite{Cohn}, there exists some compact subset $G_{i,j} \subseteq F_i \cap F_j$ such that $\lambda(G_{i,j}) = \lambda(F_i \cap F_j)$ and $\lambda(U \cap G_{i,j}) > 0$ whenever $U \cap G_{i,j} \not= \varnothing$ for open $U$. Otherwise, if $\lambda(F_i \cap F_j) = 0$, then define $G_{i,j}=\varnothing$.

For the latter case, $f_i = f_j$ is vacuously true on $G_{i,j}$. For the former case, let $x \in G_{i,j}$. Suppose that $f_i(x) \not= f_j(x)$. If we let $\eps < \displaystyle \frac{1}{2} d(f_i(x),f_j(x))$, this would imply by $0$-localization of $\Phi \circ \pi_{F_i}$ and $\Phi \circ \pi_{F_j}$ the existence of $\delta>0$ such that $\Phi(B(x;\delta) \cap G_{i,j}) \sqsubset B(f_i(x);\eps) \cap B(f_j(x);\eps) = \varnothing$. Choosing $\mu_0 := \pi_{\displaystyle B(x;\delta) \cap G_{i,j}} \lambda \not= 0$ (by the property of $G_{i,j}$ with respect to $\lambda$), we observe that $\Phi(\mu_0) = \Phi \circ \pi_{F_i}(\mu_0) = 0$, which contradicts the norm-preservation of $\Phi \circ \pi_{F_i}$ on $F_i$. It follows that $f_i = f_j$ on $G_{i,j} \subseteq F_i \cap F_j$. Having $i,j$ arbitrary, we have that outside the $\lambda$-null Borel set $N := (\cup_{i,j=1}^\infty \left( F_i \cap F_j - G_{i,j}\right)) \cup (\cap_{i=1}^\infty F_i^c)$, the functions $(f_i)_{i=1}^\infty$ coincide, allowing them to be glued to a measurable function $h : N^c \to Y$. Then extend $h$ to a measurable function $\hat{h} : X \to Y$ by $\hat{h}(x \in N) := y_0$ for some fixed $y_0 \in Y$. 

We now show that $\Phi = \hat{h}_\#$. Let $\mu \in \mc{M}_X$ and let $A_n := N^c \cap (F_n - \cup_{m<n} F_m)$. Since $\lambda(N)=0$ and $\mu \ll \lambda$, $\mu(N) = 0$. This means that 
\begin{align*}
\Phi(\mu)(B) &= \Phi\left( \pi_N \mu + \sum_{n=1}^\infty \pi_{A_n} \mu\right)(B) = 0 + \sum_{n=1}^\infty \Phi(\pi_{A_n} \mu)(B) = 0 + \sum_{n=1}^\infty {f_n}_\# (\pi_{A_n} \mu) (B) \\
&= 0 + \sum_{n=1}^\infty \mu(A_n \cap \hat{h}^{-1}(B)) =\mu(N \cap \hat{h}^{-1} (B)) + \mu(N^c \cap \hat{h}^{-1}(B)) = \hat{h}_\# (\mu) (B).
\end{align*}
\end{proof}

\section{$\eps$-Localized Transfunctions}

What about when $\Phi$ is not indentifiable with a measurable function? We consider this for uniformly localized transfunctions. Given that $\Phi: X \rsq Y$ is uniformly $\eps$-localized, can we find a measurable function $f:X \to Y$ such that $\Phi$ is uniformly $\eps$-close to $f$? If we can find such a function, then it gives a rough idea of how the transfunction behaves. In our settings, we can always find such a measurable function: in fact, it can be chosen so that $f$ is $\sigma$-simple. Can we choose a \emph{continuous} $f$ in this way? The answer is also affirmative, but it requires a more demanding setting.

\begin{theorem}\label{thm_uniformly_eps_close_measurable}
Let $X$ and $Y$ be metric spaces, with $X$ second-countable. Then every transfunction $\Phi$ which is uniformly $\eps$-localized on $X$ is uniformly $\eps$-close to some measurable function $f : X \to Y$. 
\end{theorem}

\begin{proof}
Let $\Phi: X \rsq Y$ be a uniformly $(\delta,\eps)$-localized transfunction on $X$. This means that for all $x\in X$, there exists some $y_x \in Y$ with $\Phi(B(x;\delta)) \sqsubset B(y_x;\eps)$. This choice function $x \mapsto y_x$ will be used later. Note that the collection $\{ B(x;\delta/3) : x \in X\}$ is an open cover of $X$. It follows from second-countability of $X$ that there is a countable subcover, which shall be indexed as $\{ B(x_n; \delta/3) : n \in \mb{N}\}$. For each natural $n$, let $y_n := y_{x_n}$ from the choice function above. Next we create a function $f : X \to Y$ given by $f(x) = y_n$ whenever $x \in B(x_n; \delta/3) - \cup_{m<n} B(x_m; \delta/3)$. It follows that $f$ is a $\sigma$-simple function, and therefore is measurable. Furthermore, when $f(x) = y_n$, it follows that $x \in B(x ; \delta/3) \subseteq B(x_n; \delta)$. 

Therefore, it follows that $\Phi(B(x ; \delta/3)) \sqsubset B(y_n ;\eps) = B(f(x) ;\eps)$, which shows that $\Phi$ is uniformly $(\delta/3 ,\eps)$-localized on $X$ via $f$. 
\end{proof}

We build upon the proof of Theorem \ref{thm_uniformly_eps_close_measurable} to develop the next theorem. First, we define left-translation-invariance of a metric on locally compact groups.

\begin{definition}
Let $X$ be a locally compact group with identity $e$, and let $d$ be a metric on $X$. Then $d$ is \textbf{left-translation-invariant} if $d(x,y) = d(zx,zy)$ for all $x,y,z \in X$. When the metric is understood by context, the equivalent definition is that $xB(e ;\eps) = B(x;\eps)$ for all $x\in X$ and $\eps>0$.
\end{definition}

\begin{theorem}\label{thm_uniformly_eps_close_continuous}
Let $X$ be a second-countable metrizable locally compact group with left-translation-invariant metric and let $Y$ be a normed space. Then every $\Phi$ which is uniformly $\eps$-localized on $X$ is uniformly $\eps$-close to some continuous function $g : X \to Y$.
\end{theorem}

\begin{proof}
Let $e$ denote the identity of $X$ and let $0_Y$ denote the zero in $Y$. Take from Theorem \ref{thm_uniformly_eps_close_measurable} the measurable $f: X \to Y$ as described in the previous proof with the same details. Then there exists $\alpha > 0$ such that for all $x\in X$, $B(x;\alpha)$ has compact closure. Let $x \in X$ be arbitrary. Since $\{B(x_n;\delta/3) : n \in \mb{N}\}$ covers $\overline{B(x;\alpha)}$, it follows that there is a finite subcover $\{B(x_n; \delta/3) : n \le N_x\}$ for some natural number $N_x$ depending on $x$. Therefore, $f(B(x;\alpha)) \subseteq \{y_n : n \le N_x\} \subseteq B(0_Y;M_x)$ for some real $M_x$ depending on $x$. Since $x$ was arbitrary, this means that $f$ is locally bounded.

Now let $\beta := \min\{\delta/3 , \alpha/2 \}$. Since $X$ is a locally compact group, there exists a non-zero (uniformly) continuous function $\varphi : X \to [0,\infty)$ with compact support within $B(e;\beta)$. Now choose the unique appropriately scaled left Haar measure $\kappa$ on $X$ such that $\int \varphi(u^{-1}) d\kappa (u) = 1$.

Now consider the function $g: X \to Y$ given by $g := f \ast \varphi$, the convolution of $f:X \to Y$ and $\varphi: X \to \mb{R}$ using the (vector-valued) integral

$$g(x) = f \ast \varphi (x) := \int f(t) \varphi(t^{-1}x) d\kappa (t) = \int f(xu) \varphi(u^{-1}) d\kappa(u). $$

Note that the integral above is well-defined, because $t \mapsto \varphi(t^{-1}x)$ is zero outside of $xB(e;\beta) = B(x;\beta)$ and $f$ is bounded and finitely-valued on the set $B(x;\beta)$ by an earlier argument. Also, the second equality holds due to left-invariance of $\kappa$ and the substitution $u = x^{-1}t$ which yields $xu = t$ and $u^{-1} = t^{-1}x$.

We shall now show that $g$ is continuous.  Let $x \in X$ and let $\eps > 0$ and choose some $\eta \in (0,\beta)$ with respect to uniform continuity of $\varphi$. Let $x'$ be $\eta$-close to $x$ in $X$: that is, let $x^{-1} x' \in B(e;\eta)$. This implies that $(t^{-1}x)^{-1} (t^{-1}x') = x^{-1} x' \in B(e;\eta)$ for all $t \in X$, so that $t^{-1}x$ and $t^{-1}x'$ are also $\eta$-close in $X$ for all $t \in X$. Since $d(x,x') < \alpha/2$, it follows that $B(x';\alpha/2) \subseteq B(x; \alpha)$, which means that $f(B(x'; \alpha/2)) \subseteq f(B(x; \alpha)) \subseteq B(0_Y; M_x)$. Therefore $M_x$ bounds the vectors obtained by $f$ in both $B(x;\beta)$ and $B(x';\beta)$. Then it follows that $|\varphi(t^{-1}x) - \varphi(t^{-1}x')| < \eps$ and that
\begin{align*}
||g(x) - g(x')|| = \left|\left|\int f(t) [\varphi(t^{-1}x) - \varphi(t^{-1} x')] d\kappa (t) \right|\right| \le 2M_x \eps \kappa(B(e;\beta)) .
\end{align*}
Continuity of $g$ follows since $M_x$ only depends on $x$, $\kappa(B(e;\beta))$ is a constant, and $\eps$ was arbitrary.

To show that $\Phi$ is uniformly $(\beta,\eps)$-localized via $g$, let $x\in X$ be arbitrary and let $\mu \sqsubset B(x;\beta)$. Recall that $B(x;\beta)$ is covered by $\cup_{m=1}^{N_x} B(x_m ; \delta/3)$. Notice that for every $x_m$ with $B(x_m ; \delta/3) \cap B(x; \delta/3) \not= \varnothing$ we have that $B(x;\delta/3) \subseteq B(x_m; \delta)$ which implies that $\Phi (B(x;\delta/3)) \sqsubset B(y_m ; \eps)$. 

If we denote $R := \{y_m: m \le N_x \text{ and } B(x_m;\delta/3) \cap B(x;\delta/3) \not= \varnothing\}$, and if we denote $C := \text{Conv}(R)$, the convex hull of $R$, this implies that 

$$\Phi \mu \sqsubset \bigcap_{y \in R} B(y ; \eps) = \bigcap_{y \in C} B(y; \eps) .$$

If we can show that $g(x) \in C$, then it follows from above that $\Phi$ is $(\beta,\eps)$-localized at $(x,g(x))$.

For each natural $m \le N_x$, we define $A_m := B(e;\beta) \cap x^{-1} f^{-1}(y_m)$ which is empty if $y_m \not\in R$ and we define $\displaystyle c_m := \int_{A_m} \varphi(u^{-1}) d\kappa (u)$ which is zero if $y_m \not\in R$. Then $\sum_{m=1}^{N_x} c_m = \int \varphi(u^{-1}) d\kappa (u) = 1$, and by looking at the convolution function $g$, we see that
\begin{align*}
g(x) &= \int_{B(e;\beta)} f(xu) \varphi(u^{-1}) d\kappa(u) = \int_{B(e;\beta)} \left[\sum_{m=1}^{N_x} y_m \chi_{A_m}(u)\right] \varphi(u^{-1}) d\kappa(u) \\
&= \sum_{m=1}^{N_x} y_m \int_{A_m} \varphi(u^{-1}) d\kappa(u) = \sum_{m=1}^{N_x} c_m y_m \in C .
\end{align*}

Therefore, it follows that $\Phi \mu \sqsubset B(g(x);\eps)$, meaning that $\Phi$ is uniformly $\eps$-close to $g$.
\end{proof}

\begin{corollary}\label{cor_uniformly_eps_close_reals}
Give $\mb{R}^n, \mb{R}^m$ the usual norms. Every uniformly $\eps$-localized $\Phi: \mb{R}^n \rsq  \mb{R}^m $ is uniformly $\eps$-close to some continuous function $g : \mb{R}^n \to \mb{R}^m$.
\end{corollary}

For transfunctions not uniformly localized, there is a result analogous to Theorem \ref{thm_uniformly_eps_close_measurable} with appropriate modifications of its proof.

\begin{theorem}\label{thm_eps_close_measurable}
Let $X$ and $Y$ be metric spaces with $X$ second-countable. Then every transfunction $\Phi: X \rsq Y$ which is $\eps$-localized on $X$ is $\eps$-close to some measurable function $f: X \to Y$.
\end{theorem}

\begin{proof}
We use the same framework as the proof from Theorem \ref{thm_uniformly_eps_close_measurable}. Each $x \in X$ has some associated $\delta_x > 0$ from definition of $\eps$-localization at $x$. We form the cover $\{B(x; \delta_x /3) : x \in X\}$ of $X$ which has a countable subcover $\{B(x_n; \delta_n) : n \in \mb{N}\}$, where $\delta_n := \delta_{x_n}$. We define $f(x) := y_n$ when $x \in B(x_n ; \delta_n / 3) - \cup_{m < n} B(x_m ; \delta_m / 3)$. For $x$ with $f(x) = y_n$, we have that $x \in B(x ; \delta_n / 3) \subseteq B(x_n ; \delta_n)$. This means for all $x$ with $f(x) = y_n$, we have that $\Phi (B(x;\delta_n / 3)) \sqsubset B(y_n ; \eps) = B(f(x); \eps)$, meaning that $\Phi$ is $\eps$-localized on $X$ via $f$.
\end{proof}

Alternatively, we develop a theorem analogous to the statement that continuous functions on compact sets are uniformly continuous.

\begin{theorem}\label{thm_compact_uniform}
Let $\Phi: X \rsq Y$ be a transfunction which is $\eps$-localized on $X$. Define $D_\eps (x) := \sup\{\delta > 0 : \Phi \text{ is } (\delta,\eps)\text{-localized at } x \}$. Then $D_\eps : X \to (0,\infty)$ is continuous on $X$ and if $X$ is compact, then $\Phi$ is uniformly $\eps$-localized on $X$.
\end{theorem}

\begin{proof}
Let $x_0 \in X$. Let $x \in B(x_0 ; D_\eps(x_0))$. It must follow by definition of $D_\eps$ that 
$$D_\eps(x_0) - d(x, x_0) \le D_\eps (x) \le D_\eps(x_0) + d(x,x_0);$$
this is because $B(x; D_\eps(x_0) - d(x,x_0)) \subseteq B(x_0; D_\eps(x_0)) \subseteq B(x; D_\eps(x_0) + d(x,x_0))$.

Therefore, $|D_\eps (x) - D_\eps (x_0)| \le d(x,x_0) \to 0$ as $x \to x_0$. Hence, $D_\eps$ is continuous on $X$. If $X$ is compact, then $D_\eps$ obtains its minimum, positive value on $X$; call that value $\delta_X$. Then for any positive $\delta < \delta_X$, we have that $\delta < D_\eps(x)$ for all $x \in X$, meaning that $\Phi$ is $(\delta,\eps)$-localized at every $x \in X$. This precisely means that $\Phi$ is uniformly $(\delta,\eps)$-localized on $X$.
\end{proof}

\section{Graphs of Transfunctions}

We introduce a concept analogous to the graph of a function and prove three theorems that shed some light on the nature of localized transfunctions.

\begin{definition}\label{graph}
	Let $\Phi: X \rsq Y$ be a transfunction, and let $\Gamma \subseteq X \times Y$ be measurable with respect to the product $\sigma$-algebra. We say that \textbf{$\Gamma$ carries $\Phi$}, notated as $\Phi \sqsubset \Gamma$, if for every measurable rectangle $A \times B$
$$
(A \times B) \cap \Gamma =\varnothing \quad \text{implies} \quad \Phi (A) \sqsubset B^c.
$$	
\end{definition}

Similar to how carriers of a measure describe its support, the carriers of a transfunction describe its graph. This is a generalization of the concept of a graph of a function, as indicated by the following theorem.

\begin{theorem}\label{thm_measurable_graph_carrier}
	For every measurable function $f: X \to Y$ the graph of $f$ carries $f_\#$, that is,
	$$
		f_\# \sqsubset \{(x,f(x)) : x \in X\} .
	$$
\end{theorem}

\begin{proof}
	If $(A \times B) \cap \{(x,f(x)): x \in X \} = \varnothing$, then $A \cap f^{-1}(B) = \varnothing$, so for every $\mu \sqsubset A$,
	$$
	f_\#(\mu)(B) = \mu (f^{-1} (B)) = \mu(A \cap f^{-1}(B)) = 0 .\vspace{-22 pt}
	$$
\end{proof}

We also have the reverse situation: a subset of $X\times Y$ can generate a carried transfunction.

\begin{theorem}\label{thm_transfunction_from_carrier}
	 Let $(X,\Sigma_X)$ be a measurable space and let $(Y,\Sigma_Y,\lambda)$ be a finite measure space.  If $\Gamma \subseteq X\times Y$ is a measurable set with respect to the product $\sigma$-algebra, then 
	$$
	\Phi(\mu)(B)=(\mu\times\lambda)(\Gamma \cap (X\times B))
	$$
	defines a strongly $\sigma$-additive transfunction from $X$ to $Y$ such that $\Phi \sqsubset\Gamma$.

\end{theorem}

\begin{proof}
	If $U_1,U_2,\dots \in \Sigma_Y$ are disjoint, then
	\begin{align*}
	\Phi(\mu)(\cup_{n=1}^\infty U_n)&=(\mu\times\lambda)(\Gamma \cap (X\times \cup_{n=1}^\infty U_n)) = (\mu\times\lambda)(\Gamma \cap  \cup_{n=1}^\infty(X\times U_n))\\
	&= (\mu\times\lambda)(\cup_{n=1}^\infty(\Gamma \cap  (X\times U_n))) =\sum_{n=1}^\infty (\mu\times\lambda)(\Gamma \cap  (X\times U_n))\\
	& =\sum_{n=1}^\infty \Phi(\mu)(U_n).
	\end{align*}
$\sigma$-strong additivity of $\Phi$ follows from the equality $\left(\sum_{i=1}^\infty \mu_i\right) \times \lambda = \sum_{i=1}^\infty (\mu_i \times \lambda)$. Moreover, if $(A\times B) \cap \Gamma = \varnothing$ and $\mu \sqsubset A$, then
	\[
	\Phi(\mu)(B)=(\mu\times\lambda)(\Gamma \cap (X\times B))\\
	=(\mu\times\lambda)(\Gamma \cap (A\times B))=0.
	\]
\end{proof}

Some localized transfunctions which are spatially close to measurable functions turn out to be carried by ``fat graphs". And if a transfunction has a ``fat continuous graph", then it is localized. The following theorem makes these claims precise.

\begin{theorem}\label{thm_fat_graphs}
	Let $f: X \to Y$ be measurable. If $\Phi$ is weakly $\sigma$-additive and $\eps$-localized on $X$ via $f$, then
	$$
	\Phi \sqsubset \Gamma := \bigcup_{x\in X} \left( \{x\}\times B(f(x),\varepsilon)\right).
	$$ 
	If $f$ is continuous and $\Phi \sqsubset \Gamma$, then $\Phi$ is localized on $X$ via $f$ with $E \le \eps$.
\end{theorem}

\begin{proof}
	For the forward direction, each $x \in X$ warrants some $\delta_x > 0$ such that $\Phi$ is $(\delta_x,\eps)$-localized at $(x,f(x))$. Let $(A \times B) \cap \Gamma = \varnothing$. It follows that $B \cap \left(\bigcup_{a\in A} B(f(a),\varepsilon)\right) = \varnothing$, or that $\bigcup_{a\in A} B(f(a),\varepsilon) \subseteq B^c$.
	
	$\{B(a;\delta_a) : a \in A\}$ is an open cover of $A$ with a countable subcover $\{A_n := B(a_n;\delta_n) : n \in \mb{N}\}$, where $\delta_n := \delta_{a_n}$. This ensures that $\Phi (A_n) \sqsubset B^c$ for each natural $n$. Therefore, by Proposition \ref{prop_spatial_relationship} (ii), $\Phi (\cup_{n=1}^\infty A_n) \sqsubset B^c$. It then clearly follows from $A \subseteq \cup_{n=1}^\infty A_n$ and Proposition \ref{prop_spatial_relationship} (i) that $\Phi (A) \sqsubset B^c$.
	
	For the reverse direction, let $x \in X$ and $n \in \mb{N}$ be arbitrary. Then by continuity of $f$ at $x$, there exists some $\delta$ such that $f(B(x;\delta)) \subseteq B(f(x) ; 2^{-n})$. Then it follows by definition of $\Gamma$ and by our previous argument that 
	$$
	B(x;\delta) \times B(f(x);\eps + 2^{-n})^c \cap \Gamma = \varnothing .
	$$
	Since $\Phi \sqsubset \Gamma$, it follows that $\Phi (B(x;\delta)) \sqsubset B(f(x);\eps+2^{-n})$, resulting in $\Phi$ being localized on $X$ via $f$ and that $E(x) \le \eps + 2^{-n}$ for all $x \in X$ and $n \in \mb{N}$. Since $x$ and $n$ were arbitrary, we have $E \le \eps$.
\end{proof}

\section{Markov Operators, Transport Plans and Transfunctions}

Before we introduce our next theorem, some definitions will be presented.

\begin{definition}
Let $\mu$ and $\nu$ be probability measures on $(X,\Sigma_X)$ and $(Y,\Sigma_Y)$, respectively, and let $\kappa$ be a probability measure on the product measurable space $(X\times Y, \Sigma_{X \times Y})$. We say that $\kappa$ is a \emph{transport plan} with \emph{marginals} $\mu, \nu$ if 
\begin{align*}
\kappa(A \times Y) = \mu(A) \quad \text{and} \quad \kappa(X \times B) = \nu(B) \quad \text{for all }  A \in \Sigma_X, B \in \Sigma_Y .
\end{align*}
\end{definition}

\begin{definition}\label{def_Markov_operator}
Let $\mu$ and $\nu$ be probability measures on $(X,\Sigma_X)$ and $(Y,\Sigma_Y)$ respectively. Let $\mc{L} (X,\mu)$ and $\mc{L} (Y,\nu)$ be the Banach spaces of all $\mu$-integrable functions on $X$ and all $\nu$-integrable functions on $Y$, respectively, with the usual norms. We say that a map $T : \mc{L}(X,\mu) \to \mc{L}(Y,\nu)$ is a \emph{Markov operator} if:
\begin{enumerate}[label=(\roman*)]
\item $T$ is linear with $T 1_X = 1_Y$;
\item $f \ge 0$ implies $Tf \ge 0$ for all $f \in \mc{L}(X,\mu)$;
\item $\int_X f d\mu = \int_Y Tf d\nu$ for all $f \in \mc{L}(X,\mu)$.
\end{enumerate}
\end{definition}

Notice that the definition of Markov operators depends on underlying measures $\mu$ and $\nu$ on $X$ and $Y$, respectively. Next we define some notation. If $\mu$ is a measure on $(X,\Sigma_X)$ and if $f \in \mc{L}(X,\mu)$, we define by $\int_{\square} f d\mu$ the signed measure $A \mapsto \int_A f d\mu$ and we also define the sets of measures
\begin{align*}
\mc{M}^+_\mu := \left\{\int_{\square} f d\mu : f \in \mc{L}(X,\mu) , f \ge 0\right\} \text{ and } \mc{M}_\mu := \left\{\int_{\square} f d\mu : f \in \mc{L}(X,\mu)\right\}.
\end{align*}

Notice that $\mc{M}_\mu^+$ is a set of positive measures and $\mc{M}_\mu$ is a Banach space of signed measures with respect to the variation norm, and both sets are closed under bounded sums. We now define some characteristics for transfunctions that are analogous to (ii) and (iii) from Definition \ref{def_Markov_operator}.

\begin{definition}
Let $\mu$ and $\nu$ be measures on $X$ and $Y$, respectively, and let $\Phi: \mc{M}_\mu \to \mc{M}_\nu$. We say that $\Phi$ is \textbf{positive} if $\mu \in \mc{M}_\mu^+$ implies $\Phi \mu \in \mc{M}_\nu^+$ and that $\Phi$ is \textbf{total-measure-preserving} if $(\Phi\rho)(Y)=\rho(X)$ for all $\rho \in \mc{M}_\mu$.
\end{definition}

By \cite{Taylor}, there is a bijective relationship between transport plans and Markov operators. We will show the bijective relationship between certain transfunctions and Markov operators, which will imply that all three concepts are connected. 

\begin{proposition}\label{prop_linear_bijection}
	Let $\mu$ be a measure on $(X,\Sigma_X)$. Define $b_\mu: \mc{L}(X,\mu) \to \mc{M}_\mu$ via $b_\mu f := \int_\square f d\mu$. Then $b_\mu$ is a positive linear isometry such that $b_\mu$ is strongly $\sigma$-additive on bounded sums. It follows that $b_\mu^{-1}$ has the same properties.
\end{proposition}
\begin{proof}
Positivity and linearity of integrals ensure that $b_\mu$ is positive and linear. Surjectivity of $b_\mu$ is the statement of the Radon-Nikodym Theorem. Injectivity and isometry hold because 
$$||b_\mu f|| = ||b_\mu(f^+) - b_\mu(f^-)|| = \int_X f^+ d\mu + \int_X f^- d\mu = \int_X |f| d\mu = ||f||.$$

Let $(f_i)_{i=1}^\infty$ be a sequence from $\mc{L}(X,\mu)$ such that $\sum_{i=1}^\infty ||f_i|| < \infty$. Then $\sum_{i=1}^\infty f_i \in \mc{L}(X,\mu)$ and $\sum_{i=1}^k f_i \to \sum_{i=1}^\infty f_i$ as $k \to \infty$, so by continuity of $b_\mu$, we obtain that $\sum_{i=1}^k b_\mu(f_i) \to \sum_{i=1}^\infty b_\mu (f_i)$ and $b_\mu (\sum_{i=1}^k f_i) \to b_\mu(\sum_{i=1}^\infty f_i)$ as $k \to \infty$. Since $\sum_{i=1}^k b_\mu(f_i) = b_\mu (\sum_{i=1}^k f_i)$ for all natural $k$ by linearity of $b_\mu$ it follows that $\sum_{i=1}^\infty b_\mu (f_i) = b_\mu(\sum_{i=1}^\infty f_i)$. Notice that $\sigma$-strong additivity on bounded sums only requires that $b_\mu$ is bounded and linear.

 Therefore, $b_\mu^{-1}$ is also a linear isometry by the inverse mapping theorem. Positivity of $b_\mu^{-1}$ follows since $b_\mu(\mc{L}^+(X,\mu)) = \mc{M}_\mu^+$. To see that $b_\mu^{-1}$ is $\sigma$-additive on bounded sums, replace $(f_i)_{i=1}^\infty$ with $(\rho_i)_{i=1}^\infty$ from $\mc{M}_\mu$ with $\sum_{i=1}^\infty ||\rho_i|| < \infty$ and replace $b_\mu$ with $b_\mu^{-1}$ in the argument above.
\end{proof}

Next we prove the connection between certain transfunctions and Markov operators.

\begin{theorem}\label{thm_Markov_transfunction_char}
Let $\mu$ and $\nu$ be probability measures on $(X,\Sigma_X)$ and $(Y,\Sigma_Y)$ which generate some sets $\mc{M}_X$ and $\mc{M}_Y$ respectively. Every Markov operator $T: \mc{L}(X,\mu) \to \mc{L}(Y,\nu)$ corresponds uniquely to a positive strongly $\sigma$-additive total-measure-preserving transfunction $\Phi: \mc{M}_X \to \mc{M}_Y$ with $\Phi(\mu)=\nu$ and vice versa according to the relation that 
$$
\int_B T(1_A) d\nu = \Phi(\pi_A \mu)(B)
$$ 
for all $A \in \Sigma_X, B \in \Sigma_Y$.
\end{theorem}
\begin{proof}

Let $T: \mc{L}(X,\mu) \to \mc{L}(Y,\nu)$ be a Markov operator. Define $\Phi := b_\nu \circ T \circ b_\mu^{-1}$. Since all three operators in the definition of $\Phi$ are positive and strongly $\sigma$-additive, we see that $\Phi$ is also positive and strongly $\sigma$-additive. Next, if $\rho \in \mc{M}_\mu$, then 
$$
(\Phi \rho) (Y) = b_\nu \circ T(b_\mu^{-1} \rho) (Y) = \int_Y T(b_\mu^{-1} \rho) d\nu = \int_X b_\mu^{-1} (\rho) d\mu = \rho(X)
$$ 
by the definitions of the isometries and property (iii) of $T$, so $\Phi$ is total-measure-preserving. Finally, notice that 
$$
\Phi(\pi_A \mu) (B) = (b_\nu \circ T \circ b_\mu^{-1} (\pi_A \mu))(B) = (b_\nu (T 1_A)) (B) = \int_B T(1_A) d\nu
$$ 
for all $A \in \Sigma_X$ and $B \in \Sigma_Y$, hence the relation holds.



Conversely, let $\Phi: \mc{M}_\mu \to \mc{M}_\nu$ be a positive $\sigma$-additive total-measure-preserving transfunction with $\Phi(\mu)=\nu$. Define $T := b_\nu^{-1} \circ \Phi \circ b_\mu$. Then $T(1_X) = b_\nu^{-1} \circ \Phi (b_\mu (1_X)) = b_\nu^{-1} (\Phi \mu) = b_\nu^{-1} \nu = 1_Y$. Since all three operators in the definition of $T$ are positive and strongly $\sigma$-additive, we see that $T$ is also positive and strongly $\sigma$-additive, hence linear, satisfying parts (i) and (ii) of Definition \ref{def_Markov_operator}. Next, if $f \in \mc{L}(X,\mu)$, then 
$$
\int_Y T f d\nu = \int_Y b_\nu^{-1} (\Phi \circ b_\mu f) d\nu = (\Phi(b_\mu f)) (Y) = (b_\mu f)(X) = \int_X f d\mu,
$$
 so (iii) of Definition \ref{def_Markov_operator} is met. Finally, notice that 
$$
\int_B T(1_A) d\nu = \int_B b_\nu^{-1} (\Phi \circ b_\mu (1_A)) d\nu = \Phi(b_\mu (1_A)) (B) = \Phi(\pi_A \mu) (B)
$$
 for all $A \in \Sigma_X$ and $B \in \Sigma_Y$, so the relation holds. 
\end{proof}

Note that if probability measure $\mu'$ also generates $\mc{M}_\mu$, and if we define $\nu'=\Phi(\mu')$, then the same $\Phi: \mc{M}_X \to \mc{M}_Y$ corresponds to a Markov operator $T' : \mc{L}(X,\mu') \to \mc{L}(Y,\nu')$ and it corresponds to a transport plan $\kappa'$ with marginals $\mu',\nu'$. Therefore $T$ and $T'$ are different Markov operators, $\kappa$ and $\kappa'$ are different plans, yet they follow the same ``instructions'' encoded by $\Phi$. In this regard, $\Phi$ is a global way to describe a transportation method independent of marginals.

\printbibliography


\end{document}